\documentclass[11pt]{article}
\usepackage{amsmath,amssymb,amsthm}
\usepackage{graphicx}
\graphicspath{{C:/Users/SAW/Desktop/MyFiles/CBMS/CBMSbook/}{MyFiles}}

\newcommand{\epp}{_{\epsilon}}
\newcommand{\Ge}{G_{\epsilon}}
\newcommand{\Fe}{F_{\epsilon}}
\newcommand{\Se}{S_{\epsilon}}
\newcommand{\Ne}{N_{\epsilon}}
\newcommand{\Pe}{P_{\epsilon}}
\newcommand{\Ke}{K_{\epsilon}}

\newtheorem{definition}{Definition}

\newtheorem{theorem}[definition]{Theorem}
\newtheorem{corollary}[definition]{Corollary}
\newtheorem{proposition}[definition]{Proposition}

\begin{document}

\title{Schiffer variations and Abelian differentials\footnote{2010{\em Mathematics Subject Classification.} Primary 32G15, 14H15}}         
\author{Scott A. Wolpert}        
\date{\today}          
\maketitle
\begin{abstract}
	Deformations of compact Riemann surfaces are considered using a \v{C}ech cohomology sliding overlaps approach. Cocycles are calculated for conformal cutting and regluing  deformations at zeros of Abelian differentials.  Deformations fixing the periods of a differential and deformations splitting zeros are considered.  A second order deformation expansion is presented for the Riemann period matrix. A complete deformation expansion is presented for Abelian differentials.  Schiffer's kernel function approach for deformations of a Green's function is followed.  
\end{abstract}

\section{Introduction}
We consider compact Riemann surfaces of positive genus with accompanying Abelian differentials.  An Abelian differential provides a translation surface structure and a period functional on singular homology.  We are interested in geometrically defined deformations of surfaces and differentials - deformations defined by cutting and edge regluing. Our purpose is to study deformations by a \v{C}ech cohomology style sliding overlaps approach. The approach does not involve  potential theory or the  $\bar{\partial}$ operator.   

Ahlfors \cite{Ahanal} and Rauch \cite{Rauch} understood the complex structure on Teichm\"{u}ller space by considering families of compact Riemann surfaces as varying branched covers of $\mathbb P^1$.  They showed that the Riemann period matrix varies holomorphically as a function of the branch points and that this property characterizes the complex structure on the complement of the locus of hyperelliptic surfaces.  Rauch developed his celebrated variational formula as part of the study \cite{Rauch}.  Schiffer also developed an approach for variations of the period matrix based on his deformation of removing a disc and regluing by a function defined in a neighborhood of the boundary \cite[Chapter 7]{SchS1}.  Schiffer's approach predates the advent of \v{C}ech cohomology and the cutting and pasting deformations now studied in Teichm\"{u}ller dynamics. 

The present investigation is motivated by several goals.  The first is to present formulas suitable for cutting and pasting deformations. The second is to present formulas for higher derivatives of the period mapping.  The third is to present formulas for relative deformations, deformations fixing the periods of a differential. Focus is given to deformations of zeros of  differentials, including splittings of zeros.    

In Sections \ref{slit} and \ref{4var}, we  use scalings of classical conformal slit mappings to describe cut and reglue deformations at a zero of an Abelian differential.  Explicit families are described, including splitting higher order zeros.  The deformations are local; an Abelian differential is only deformed in a neighborhood of a zero.   The periods of the differential are not changed by the constructions. Each deformation varies the lengths of horizontal and vertical trajectories ending at a zero. The first and second \v{C}ech deformation cocycles with values in vector fields are computed; see formulas (\ref{SSvar}) through (\ref{SKvar}). A slit mapping is the Riemann mapping for the complement (including the point at infinity) of a configuration of line segments at the origin.  We use combinations of scalings of slit mappings to prescribe deformations.  The given families describe cutting open along line segments and regluing the resulting edges by a new pattern.  If the construction is performed in the coordinate for the normal form of a differential, then the differential reglues to a new differential.  The line segments  correspond to trajectory segments of the differential.  Two of the constructions relate directly to the Kontsevich-Zorich cutting and pasting deformations of a zero \cite[Section 4.2, Figure 2]{KontZor}.

Green's functions are the basic analytic tool for Riemann surfaces.  Schiffer gives deformation formulas for Green's functions and obtains Abelian differentials as integrals of Green's functions.  In Sections \ref{sec4} and \ref{sec5}, we  follow Schiffer's approach for Green's functions and Abelian differentials to derive deformation expansions \cite[Chapters 3, 4 and 7]{SchS1}.  Schiffer begins with $d\Omega_{q_0q_1}$, the Abelian differential of the third kind, periods with vanishing real parts, with a pole of residue $-1$ at $q_0$ and $+1$ and at $q_1$.   The multivalued function  $\Omega_{q_0q_1}$ is the indefinite integral of the differential $d\Omega_{q_0q_1}$.  The {\em double pole Green's function} is defined as
$$
V(p,p_0;q,q_0)\,=\,\Re\{\Omega_{qq_0}(p)\,-\,\Omega_{qq_0}(p_0)\}
$$
and the Abelian kernel as
$$
\Lambda(p,q)\,=\,-\frac{1}{\pi}\frac{\partial^2\Omega_{qq_0}(p)}{\partial p\partial q}.
$$
The Abelian differential dual to a cycle is the integral of $\Lambda$ over the cycle.  Deformation formulas are given in terms of the Green's function, the Abelian differentials of the third kind and the Abelian kernel.  
 
A deformation of a compact surface $R$ is described as follows.  The construction is for a local coordinate $z$ with domain $U$ and a curve $\gamma$ bounding a disc in the domain of $z$.  
Let $r(z)$ be a holomorphic function in $z$ with domain a neighborhood of $\gamma$.     Provided $r(z)$ is suitably small, the data defines a new Riemann surface $R^*$ given by attaching the exterior of $\gamma$ (the complement of the disc bound by $\gamma$) to the interior of $\gamma^*=\gamma+r(\gamma)$ by identifying $z(p)$ on $\gamma$ to $z(p)+r(z(p))$ on $\gamma^*$.  

Our considerations begin with Schiffer's exact relation for the variation of the Green's function; see Theorem \ref{exactform}. Expansions are derived from the exact relation.  In Theorem \ref{main} the exact relation is combined with the definition of the deformed structure and Taylor's theorem to give a second order expansion for the Green's function.  Then in Corollary \ref{derivperi}, the relation between the Abelian kernel $\Lambda$ and a basis of differentials is used to give a second order expansion for the Riemann period matrix.  If the deformation cocycle has coefficient a rational function then the expansion is evaluated in terms of the values and derivatives of the Abelian kernel and the basis of differentials.  In Corollary \ref{formu}, a local coordinate $z$ with $z(p)=0$ and the particular deformation cocycle $(\epsilon\frac{a}{z}+\frac{\epsilon^2}{2}\frac{b}{z^m})\frac{d}{dz}$ for $a,b\in \mathbb C$ and $m\in \mathbb N$ are considered. We find the second order expansion for the Riemann period matrix
\begin{multline*}\label{mainform}
\Gamma_{\mu\nu}^*\,=\,\Gamma_{\mu\nu}\,+\,\frac{\epsilon\pi i}{2} a\,\omega_{\mu}(p)\omega_{\nu}(p)\,-\,\frac{\epsilon^2\pi^2 i}{2}a^2\lambda(p)\omega_{\mu}(p)\omega_{\nu}(p)\\
+\frac{\epsilon^2\pi i}{4}\bigg(a^2\omega_{\mu}'(p)\omega_{\nu}'(p)\,+\,b\frac{1}{(m-1)!}\frac{d^{m-1}}{dz^{m-1}}\big(\omega_{\mu}(z)\omega_{\nu}(z)\big)\big\vert_{z=p}\bigg)\,+\,O(\epsilon^3),
\end{multline*}
where for a canonical homology basis, the differentials $\{\omega_{\mu}\}$ are dual in the integral pairing to the period functionals and the differentials are evaluated in the local coordinate $z$ and $\lambda$ is a local coordinate  regularization of $\Lambda$. The first order expansion is essentially Rauch's formula \cite{Rauch} and was already given by Schiffer in \cite[Section 7.8]{SchS1}.  We note from Corollary \ref{derivperi} that if $z$ is the coordinate for the normal form $z^mdz$ at a zero of order $m$ at point $p$ for a differential $\omega$, then the deformation cocycles
$$
\frac{1}{z}\frac{d}{dz},\,\dots\,,\frac{1}{z^m}\frac{d}{dz}
$$
in a neighborhood of $p$, give vanishing of the first variation of the periods of $\omega$.  In Proposition \ref{defbasis}, we find for a non trivial differential $\omega$, a basis for the infinitesimal deformations of the surface is given by any $2g-3$ Schiffer deformations of the (possibly multiple) zeros of $\omega$ and $g$ Schiffer deformations at a general point.

We follow Schiffer's approach in Proposition \ref{omegform} to find the complete expansion for the variation of an Abelian differential.  Again if the deformation cocycle has coefficient a rational function, the expansion is evaluated in terms of the values and derivatives of the Abelian kernel and the initial differential. For the local coordinate $z$ with $z(p)=0$, and the deformation cocycle $\epsilon\frac{a}{z}\frac{d}{dz}$, we find the second order expansion for a basis differential  
 \begin{multline*} 
 \omega_{\mu}^*(q)\,=\, \omega_{\mu}(q)\,-\,\epsilon\pi a\Lambda(q,p)\omega_{\mu}(p)
\\
 +\,\frac{\epsilon^2}{2}\big(2\pi^2 a^2\Lambda(q,p)\lambda(p)\omega_{\mu}(p)\,-\,\pi a^2\frac{\partial}{\partial p}\Lambda(q,p)\frac{\partial}{\partial p}\omega_{\mu}(p)\big)\,+\,O(\epsilon^3),
 \end{multline*}
for $q\ne p$ and quantities evaluated in the variable  $z$.  
The present expansions for Abelian differentials and the Riemann period matrix can be compared to the complete expansions of Karpishpan\cite{Karp}, Yin \cite{Yin}, Zhao-Rao \cite{ZhRa}, Liu-Zhao-Rao \cite{LiZhRa} and Yamada \cite{AYam}. Karpishpan shows that the differentials of the map are induced by cup products involving the Kodaira-Spencer class and an Archimedean cohomology.  The remaining authors use $\bar{\partial}$-methods and give expansions with iterated integrals of the Green's function acting on one-forms. The authors investigate  the period mapping of  Teichm\"{u}ller space to the Siegel upper half space.  Yin considers the relation of the image to geodesics of the Siegel metric.  Zhao-Rao develop formulas for the induced metric, its second fundamental form and its curvature. Liu-Zhao-Rao investigate the Torelli theorems. Yamada follows Schiffer's approach of Green's functions and the Abelian kernel to give degeneration expansions for Abelian differentials and the period matrix. 

Colombo and Frediana develop complete formulas for the induced Siegel metric, its second fundamental form and its curvature in terms of Schiffer variations \cite{ColFr}. Consider a basis of Abelian differentials $\{\alpha_{\mu}\}$ dual to the {\em A cycles} of a canonical homology basis with $\Pi_{\mu\nu}$ the corresponding Riemann period matrix.  For a local coordinate $z$ with points $p,q$ in its domain, the Siegel pairing of  Schiffer variations is 

$$
\big\langle \frac{1}{z-z(p)}\frac{d}{dz},\frac{1}{z-z(q)}\frac{d}{dz}  \big\rangle\,=\,4\pi^2\big(\sum_{\mu,\nu}\alpha_{\mu}(p)\,\big(\Im \Pi\big)^{-1}_{\mu\nu}\, \overline{\alpha_{\nu}(q)}\big)^2, 
\ \cite{ColFr}.
$$

The present investigation is motivated by the significant and current research on families of Abelian differentials.  Particular motivation comes from the work of Eskin-Mirzakhani-Mohammadi on the dynamics of the $\operatorname{SL(2;\mathbb R)}$ action on differentials \cite{EMM}, Grushevsky-Krichever on isoperiodic families of meromorphic differentials and the topology of the moduli space \cite{GruKri, GruKrifol}, Kontsevich-Zorich on the topology of families of differentials with prescribed zero orders \cite{KontZor}, McMullen on the special structure of isoperiodic families of differentials \cite{McMbill, McMiso} and especially Calsamiglia-Deroin-Francaviglia on Schiffer variations and the general structure of isoperiodic families \cite{Der}. An overview of recent results is given in the expository article of Alex Wright 
\cite{Wright}.  It is my pleasure to thank Alex Wright for conversations and posing the question of evaluating the variation of the period matrix.

\section{Slit mappings}\label{slit}
The analytic function
\begin{equation*}
f(z)\,=\,z\prod^n_{\nu=1}(1-\frac{e^{-i\theta_{\nu}}}{z})^{\alpha_{\nu}}
\end{equation*}
for $\alpha_{\nu}$ positive, with $\alpha_1+\cdots+\alpha_n=2$, and $\theta_1<\cdots<\theta_n$ is the Riemann mapping from $\{|z|>1\}\subset\hat{\mathbb{C}}$ to the complement in $\hat{\mathbb C}$ of an arrangement of radial slits at the origin \cite[Theorem 2.6, Example 2.1]{Pomm}.  The mapping is asymptotic to the identity at infinity $f(z)=z+O(1/|z|)$. The radial slits give angular sectors at the origin of measures $\pi\alpha_{\nu}, \nu=1,\cdots,n$. The angles $\theta_{\nu}, \nu=1,\cdots,n,$ are the preimages of the origin on the unit circle $\{|z|=1\}$ and also determine the lengths of the radial slits.  The reciprocal function $1/f(z)$ defines a mapping to a star like domain; a domain convex relative to the origin.  The mapping represents a special case of a Schwarz-Christoffel mapping - the configuration of slits can be considered as a polygon with empty interior and pairs of consecutive sides coinciding.  

We are interested in four particular mappings. The first is the classical slit mapping
\begin{equation}\label{koebemap}
S(z)\,=\,z\,+\,\frac1z\,,
\end{equation}
the Riemann mapping from the exterior of the unit disc (or equivalently the unit disc) to the complement of $\{w\mid -2\le \Re w\le 2,\Im w=0\}$. 
The second is for $n$ a positive integer, the mapping
\begin{equation}\label{star}
N(z)\,=\,z(1-\frac{1}{z^n})^{2/n}\,,
\end{equation}
the Riemann mapping from the exterior of the unit disc to the complement of the regular $n$-star at the origin.  The roots of unity are the preimages of the origin under the mapping. The midpoints between consecutive roots are the preimages of the tips of the star. The mapping for $n=2$ is a conjugation by a rotation of the classical slit map (\ref{koebemap}).  
\begin{figure}
\centering
\includegraphics[width=6cm]{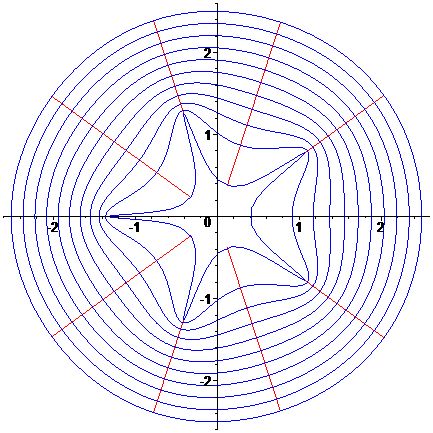}
\caption[The 5 star]{The $5$-star. The contour curves are the images by $N(z)$ of the polar coordinate contours from the domain $\{|z|>1\}$.}
\label{fig:5spike}
\end{figure}

The third is the special mapping
\begin{equation}\label{sectormap}
P(z)\,=\,z(1+\frac1z)^{4/3}(1-\frac1z)^{2/3}\,,
\end{equation}
the Riemann mapping from the exterior of the unit disc to the complement of equal length slits, each at angle $\pi/3$, with the positive real axis.  
\begin{figure}
\centering
\includegraphics[width=6cm]{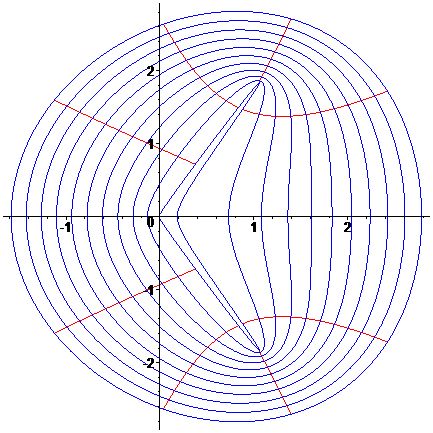}
\caption[The 1/3 sector]{The special $2\pi/3$ sector mapping. The contour curves are the images by $P(z)$ of the polar coordinate contours from the domain.}
\label{fig:23sector}
\end{figure}

\noindent The fourth for a parameter $0\le \theta\le \pi/2$ is 
\begin{equation} \label{skeanmap}
K(z)\,=\,z(1-\frac{e^{2i\theta}}{z^2})^{1/2}(1-\frac{e^{-2i\theta}}{z^2})^{1/2}\,,
\end{equation}  
the Riemann mapping from the exterior of the unit disc to the complement of a  horizontal-vertical skean.  For $\theta$ small, the skean has a shorter horizontal segment and for $\theta$ close to $\pi/2$, the skean has a shorter vertical segment.  The points $e^{i\theta},e^{i(\pi-\theta)},e^{i(\theta-\pi)},e^{-i\theta}$ are the preimages of the origin on the unit circle.  For $\theta=\pi/4$, the mapping is a conjugation by a rotation of the mapping $N(z)$ for $n=4$. 

\begin{figure}
\centering
\includegraphics[width=6cm]{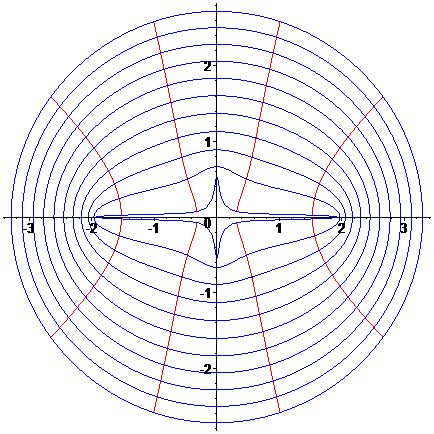}
\caption[The horizontal skean.]{The skean with $\theta=1.3$. The contour curves are the images by $K(z)$ of the polar coordinate contours from the domain.}
\label{fig:skean}
\end{figure}

Mappings between slit domains are given by considering the composition of one map and the inverse of a second map.  For a map $F(z)$ asymptotic to the identity at infinity (such as the above maps) then the scaling $\epsilon\, F(z/\epsilon)$ is a variation of the identity map for $\epsilon$ small.  In preparation for considering compositions and scalings, we note elementary formulas.
\begin{proposition}
	\label{elem}
Let $w=f(z)$ be holomorphic in a neighborhood of the origin, fixing the origin, and with a local inverse function.  The initial derivatives of the inverse are $(f^{-1})'(0)=(f'(0))^{-1}$ and $(f^{-1})''(0)=-f''(0)(f'(0))^{-3}$. Let $\Fe(z)$ and $\Ge(z)$ be holomorphic for $\epsilon$ small and $\{|z|>c\}$, with $F_0(z)=G_0(z)=z$.  For a constant $a$, the initial parameter derivatives of $\Ge(a\Fe(z))$ are
	\begin{equation}\label{first}
	\frac{d}{d\epsilon}\Ge(a{\Fe}(z))\big\vert_{\epsilon=0}\,=\,\dot{\Ge}(az)\,+\,a\dot{\Fe}(z)
	\end{equation}
	and
	\begin{equation}\label{second}
	\frac{d^2}{d\epsilon^2}\Ge(a\Fe(z))\big\vert_{\epsilon=0}\,=\,\ddot{\Ge}(az)\,+\,2\big(\frac{d}{dz}\dot{\Ge}(az)\big)a\dot{\Fe}(z)\,+\,a\ddot{\Fe}(z)
	\end{equation}
 where $\,\dot{}\,$ indicates an $\epsilon$ derivative evaluated at the origin.
\end{proposition}
 \begin{proof}
 	The derivative formulas for $f^{-1}$ are immediate.  The first and second parameter derivatives of $\Ge(a\Fe(z))$ are
 	$$ \dot{\Ge}(a\Fe(z))\,+\,\big(\frac{d}{dz}\Ge(a\Fe(z))\big) a\dot{\Fe}(z)
 	$$
 	and
 	\begin{multline*} 	 	
 	\ddot{\Ge}(a\Fe(z))\,+\,2\big(\frac{d}{dz}\dot{\Ge}(a\Fe(z))\big)a\dot{\Fe}(z) \\ 
 	+\,\big(\frac{d^2}{dz^2}\Ge(a\Fe(z))\big)\big(a\dot{\Fe}(z)\big)^2\,+\,\big(\frac{d}{dz}\Ge(a\Fe(z))\big)a\ddot{\Fe}(z),
    \end{multline*}
    where $\dot{ }\,$ simply indicates the $\epsilon$ derivative.  The desired formulas follow from evaluating $\epsilon=0$ with the initial condition $F_0(z)=G_0(z)=z$.  	
 \end{proof}
 
 The first application is for $\Ge(z)$ an invertible function of $z$ for small $\epsilon$.  For $\Fe(z)$ the inverse function, then $\Ge(\Fe(z))=z$ and the parameter derivatives of the composition vanish.  The initial derivative formulas (with $a=1$) are
 \begin{equation}\label{inv1}
 \dot{F}_0(z)\,=\,-\dot{G}_0(z)
 \end{equation}
 and
 \begin{equation}\label{inv2}
 \ddot{F}_0(z)\,=\,-\ddot{G}_0(z)\,+\,2\dot{G}_0(z)\frac{d}{dz}\dot{G}_0(z).
 \end{equation}
 
 \section{The four variations}\label{4var}
 
 The Schiffer variation (deformation) is given by removing a (coordinate) disc and reattaching by a function holomorphic near the disc boundary.  The variation is trivial if the attaching function is holomorphic on the disc and otherwise can be non trivial.  In particular, if $\gamma$ is a simple closed curve in a $z$ coordinate disc and $r(z)$ is a holomorphic function in a neighborhood of $\gamma$, then the interior of the curve can be attached by identifying $\gamma(t)$ with $\gamma(t)+\epsilon r(\gamma(t)),$ for $\epsilon$ small \cite[Section 7.8, Deformation by attaching a cell]{SchS1}.  The construction gives a family in $\epsilon$.  The variation can also be understood by the overlap rule, $z$ on a neighborhood of $\gamma$ is identified with $w=z+\epsilon r(z)$ on a neighborhood of $\gamma$.  The values of $r(z)$ are understood as displacements relative to the domain variable; the $\epsilon$-derivative of the overlap identification is the displacement vector field $r(z)\frac{d}{d z}$ on the overlap.  This {\it sliding overlaps} description aligns with the Kodaira-Spencer \v{C}ech cohomology formulation of infinitesimal deformations in $\Check{H}^1(\Theta),$ for $\Theta$ the sheaf of holomorphic vector fields  \cite[I, Section 5]{KoSp}. The Schiffer variation predates the \v{C}ech description by some twenty years.  
 
  The Serre duality pairing of $\Check{H}^1(\Theta)$ with the space  of holomorphic quadratic differentials can be evaluated using the $\bar{\partial}$-resolution of the sheaf of holomorphic vector fields \cite{GunRS}.  For the above vector field and $U,V$ an open cover with $\gamma$ a positively oriented core curve of the topological annulus $U\cap V$, the duality pairing is 
 $$
 \int_{\gamma} r(z)\varphi
 $$
 for a holomorphic quadratic differential $\varphi$ expressed in terms of the coordinate $w$.

 We now use slit mappings to define deformations with a geometric cutting-and-edge-regluing description.  The description is tailored to study deformations of {\it flat structures} on Riemann surfaces.  In particular for appropriate deformations at the zero of an Abelian differential, the differential will reglue to an Abelian differential on the new surface.  Equivalently, flat structures are deformed to flat structures by cutting and edge regluing.  In particular for $F(z)$ and $G(z)$ Section \ref{slit} slit mappings, then the scalings $\Fe(z)=\epsilon F(z/\epsilon)$ and $\Ge(z)=\epsilon G(z/\epsilon)$ are holomorphic for $\epsilon$ small and $z$ in the complement of a compact set.  Each of $\Fe(z)$ and $\Ge(z)$ is a variation of the identity fixing infinity.  The composition $\Fe(\Ge^{-1}(w))$ has the following geometric description. Begin with the plane $\mathbb C$ and cut open the slits for $\Ge$; map the configuration of slit edges to the unit circle by $\Ge^{-1}$; map the unit circle to the slits of $\Fe$ by $\Fe$ and finally glue the adjacent edges.  The maps are Riemann mappings for the slit complements and the cut open slit domains are complements of degenerate analytic polygons.  The maps are homeomorphisms of the configurations of cut open slits.  The deformation is given by identifying $w$ with $\zeta=\Fe(\Ge^{-1}(w))$. The cutting and gluing pattern provides that the map is not continuous on $\mathbb C$.  The cutting and gluing is on a set of scale $\epsilon$.  For $\epsilon$ small, away from the origin, the identification of $w$ to $\Fe(\Ge^{-1}(w))$ is holomorphic and close to the identity.  Formally speaking, $\Fe(\Ge^{-1}(w))$ is a holomorphic family varying from the identity.  We will compute its first and second variations (as \v{C}ech $1$-cocycles valued in vector fields)  and show that the families deform special Abelian differentials and flat structures.
 
 Consider that the parameter $w$ is a local coordinate for a Riemann surface $R$.  Consider an open cover of $R$ consisting of two open sets: $U$ a relatively compact subset of the domain of $w$ and $V$ the complement of a relatively compact subset of $U$.    Consider that the point $w=0$ is contained in $U$, not in $V$ and that $U\cap V$ is a topological annulus.  We write $\zeta$ for the local coordinate $w$ restricted to its domain in $V$. For $\zeta=A_{\epsilon}(w)$ a holomorphic family varying from the identity ($A_0(w)=w$), such as the family $\Fe(\Ge^{-1}(w))$ above, a family of Riemann surfaces $R_{\epsilon}$ is defined by: $p\in U$ is identified with $q\in V$ provided $\zeta(q)=A_{\epsilon}(w(p))$.  The infinitesimal variation of the family is described by the displacement vector field $\frac{d}{d\epsilon} A_{\epsilon}(w)\frac{d}{dw}$ on $U\cap V$.

 We begin with considering the scaling of the classical slit mapping $S(z)$ (see (\ref{koebemap}))
 \begin{equation}\label{Se}
 w\,=\,S_{\epsilon}(z)\,=\,\epsilon S(z/\epsilon)\,=\,z\,+\,\frac{\epsilon^2}{z}.
 \end{equation}
 The domain of $\Se(z)$ is $\{|z|>\epsilon\}$ and the range is the complement of $\{w\mid -2\epsilon \le \Re w\le 2\epsilon, \Im w=0\}$.  For $\theta$ the natural angle parameter of the radius $\epsilon$ circle, the boundary mapping is $\Re w=2\epsilon\cos\theta$.  
The variation of $S_{\epsilon}(z)$ from the identity is {\it linear} in $\epsilon^2$. The first  variation cocycle  is
  \begin{equation}
  \label{SSvar}
  \frac1z\frac{d}{dz}
  \end{equation}
  and the second variation cocycle is trivial.    
 
 Deformations can be given by cutting and regluing the pattern of trajectories at zeros of an Abelian differentials.  The first composition to consider is for an integer $n>1$, $e_n=e^{\pi i n}$ an $n^{th}$ root of $-1$, the map  $e_n^{-1}\Ne(e_n\Ne^{-1}(w))$, for the scaling of the map (\ref{star}).  The map is a variation from the identity.  The composition is a family in $\epsilon^n$, since $\Ne(z)$ is a family in $\epsilon^n$.  Using the binomial series we find the expansion
 $$
 \Ne(z)\,=\,z\big(1\,-\,\frac2n\frac{\epsilon^n}{z^n}\,+\,\frac1n (\frac2n-1)\frac{\epsilon^{2n}}{z^{2n}}\,+\,O(\epsilon^{3n})\big).
 $$
In particular for the initial $\epsilon^n$ variations, we have
$$
\dot{\Ne}(z)\,=\,-\frac2n\frac{1}{z^{n-1}}
\quad\mbox{and}\quad \ddot{\Ne}(z)\,=\,\frac2n(\frac2n-1)\frac{1}{z^{2n-1}}.
$$
By formulas (\ref{inv1}) and (\ref{inv2}) the initial $\epsilon^n$ variations of $\Ne^{-1}$ are
$$
\dot{\Ne^{-1}}(w)\,=\,\frac2n\frac{1}{w^{n-1}}\quad\mbox{and}\quad \ddot{\Ne^{-1}}(w)\,=\,\frac{4-6n}{n^2}\frac{1}{w^{2n-1}}.
$$
Applying Proposition \ref{elem}, we find the first and second variation  cocycles
\begin{equation}
\label{NNvar} 
\dot{e_n^{-1}\Ne(e_n\Ne^{-1}}(w))\,=\,\frac{4}{nw^{n-1}}\frac{d}{dw}
\  \mbox{and}\ \   \ddot{e_n^{-1}\Ne(e_n\Ne^{-1}}(w))\,=\,\frac{16(1-n)}{n^2w^{2n-1}}\frac{d}{dw}.
\end{equation}

We consider the geometric description of the mapping.  The symmetries of $\Ne$ are generated by three transformations: conjugation by rotation through angle $2\pi/n$, reflection in the real axis and reflection in the line $\arg w=\pi/n$.
The $n^{th}$ roots of unity map to the origin by $\Ne$ and the rotation by $\pi/n$ of the $n^{th}$ roots map to the tips of the regular $n$-star.   
The slits of the star are at the angles $(1+2k)\pi/n$ for $k=1,\dots,n$. Let $\bf{a}$ be a slit of the star and $\bf{a'}$ the counterclockwise consecutive slit of the star.  The mapping $\Ne(e_n\Ne^{-1}(w))$ identifies the counterclockwise edge of $\bf{a}$ with the clockwise edge of $\bf{a'}$.  The symmetry of $\Ne$ provides that the identification is the reflection across the bisector of the intermediate sector.  The composition interchanges the origin and the endpoints of the star.  The Abelian differential $\omega=-w^{n-1}dw$ is positive on the slit of the star and is conjugated by the reflection across the bisector of an intermediate sector.   By Schwarz reflection, the pushforward (the map is invertible) of $\omega$ by $\Ne(e_n\Ne^{-1})$ extends to be holomorphic on $\mathbb C$ except possibly at the images of the origin and the images of the star tips.  The local form of the map is $v=u^n$ at the vertex of a star sector. Accordingly $dv=nu^{n-1}du$ and at the image of the vertex the pushforward of $\omega$ is holomorphic and nonzero.  The local form is $v^n=u$ at a tip of the star.   Accordingly $nv^{n-1}dv=du$ and at the image of a tip the pushforward of $\omega$ is holomorphic with a zero of order $n-1$.  The map $\Ne(e_n\Ne^{-1})$ pushes forward $\omega$ to an Abelian differential with an order $n-1$ zero at the origin.  The local construction in a small disc is based on using the coordinate for the normal form of the differential as the coordinate for the map.  The periods of an Abelian differential are not changed by a local construction in a disc; the deformation is isoperiodic.  Since the slits are a null set for the Hermitian square of $\omega$, the integral norm of $\omega$ is given by integrating over the slit complement and the norm is not changed by the mapping.  For a compact Riemann surface, the second observation also follows from the first by the Riemann bilinear relations.     
 
 The second composition to consider is $\Se(\Pe^{-1}(w))$ for the scaling of the maps (\ref{koebemap}) and (\ref{sectormap}).  The composition maps the complement of equal length slits, at angles $\pm\pi/3$ with the positive real axis, to the complement of $\{\zeta\mid -2\epsilon\le\Re \zeta\le 2\epsilon,\Im \zeta=0\}$.  The composition is a family in $\epsilon$ since $\Se$ is a family in $\epsilon^2$ and $\Pe$ is a family in $\epsilon$.  Using the binomial series for $(1+\epsilon/z)^{4/3}$ and $(1-\epsilon/z)^{2/3}$, we find the expansion
 $$
 \Pe(z)\,=\,z\,+\,\frac23\epsilon\,-\,\frac79\frac{\epsilon^2}{z}\,+\,O(\epsilon^3).
 $$
 For the initial $\epsilon$ variations we have
 $$
 \dot{\Pe}(z)\,=\,\frac23\quad\mbox{and}\quad\ddot{\Pe}(z)\,=\,-\frac{14}{9}\frac1z.
 $$
 By formulas (\ref{inv1}) and (\ref{inv2}) the initial variations of $\Pe^{-1}$ are
 $$
 \dot{\Pe^{-1}}(w)\,=\,-\frac23\quad\mbox{and}\quad \ddot{\Pe^{-1}}(w)\,=\,\frac{14}{9}\frac{1}{w}.
$$
Recalling the initial variations $\dot{\Se}=0$ and $\ddot{\Se}=2/z$, applying Proposition \ref{elem}, we find the  first and second variation cocycles  
\begin{equation}\label{SPvar}
-\frac23\frac{d}{dw}\qquad\mbox{and}\qquad\frac{32}{9}\frac1w\frac{d}{dw}.
\end{equation} 
The first cocycle is the \v{C}ech coboundary of the assignment $-2/3$ on $U$ and $0$ on $V$; the first cocycle is the trivial variation.  

We consider the geometric description of the mapping $\Se(\Pe^{-1})$.   The range of $\Pe$ has a positive slit $\bf{p}$ where $\Im \zeta\ge0$ and a negative slit $\bf{n}$ where $\Im \zeta\le 0$.  
Denoting the upper and lower edges of each slit by $\pm$s, we have the quadrilateral of consecutive edges $\bf{p^-}, \bf{p^+}, \bf{n^-}, \bf{n^+}$ maps by $\Pe^{-1}$ to the radius $\epsilon$ circle and by $\Se$ to the slit $\{\zeta\mid -2\epsilon\le\Re \zeta\le 2\epsilon,\Im\zeta=0\}$.  The mapping identifies $\bf{p^+}$ with $\bf{n^-}$ and $\bf{p^-}$ with $\bf{n^+}$.  Points symmetric with respect to the real axis are identified. The Abelian differential $\omega=-w^2dw$ is conjugated by the reflection in the real axis and is positive on the slits and the rays $\arg z=\pm \pi/3$. By Schwarz reflection, it follows that the pushforward of $\omega$ by $\Se(\Pe^{-1})$ extends to be holomorphic on $\mathbb C$ except possibly at the images of the vertices of the quadrilateral.  The local form of the map is $v=u^3$ at the vertex of the sector with angle $2\pi/3$.  Accordingly $dv=3u^2du$ and at the image the pushforward of $\omega$ is holomorphic and nonzero.  At the tips of the slits, the local form of the map is $v^2=u$ (a local disc maps to a local half disc). Accordingly $2vdv=du$ and at the image the pushforward is holomorphic with a simple zero at the image point (the slit tips are identified by the map).  Finally at the vertex of the $4\pi/3$ sector the local form is $v^2=u^3$.  Accordingly $2vdv=3u^2du$ and at the image the pushforward of $\omega$ is holomorphic with a simple zero.   The map $\Se(\Pe^{-1})$ pushes forward $\omega$ to an Abelian differential with simple zeros at distance scale $\epsilon$.  The map $\Se(\Pe^{-1})$ {\em splits a double zero}.  The local construction in a small disc is based on using the coordinate for the normal form of the differential as the coordinate for the map.  The periods of an Abelian differential are not changed by a local construction in a disc; the deformation is isoperiodic.  Since the slits are a null set for the Hermitian square of $\omega$, the integral norm of $\omega$ is given by integrating over the slit complement and the norm is not changed by the mapping.  For a compact Riemann surface, the second observation also follows from the first by the Riemann bilinear relations. 

The above splitting a double zero is close to the Kontsevich-Zorich  breaking up a zero \cite[Section 4.2, Figure 2]{KontZor}.  In particular start with a disc neighborhood of a double zero tiled by six equi angular sectors. Label the sectors by the angle from the positive axis; sector $k$ is $(k-1)\pi/3\le\arg w\le k\pi/3, 1\le k\le 6$.  The sectors are Kontsevich-Zorich's six half discs.  The deformation identifies the scale $\epsilon$ (corresponding to the Kontsevich-Zorich displacement $\delta$) segments of $\bf{p}^-$ and $\bf{n}^+$ from the boundaries of sectors $1$ and $6$ to form an additional segment on the positive axis.  The deformation identifies the segments $\bf{p}^+$ and $\bf{n}^-$ from the boundaries of sectors $2$ and $5$ to form a common boundary  between the sectors.  The present deformation does not have the left-right symmetry of Kontsevich-Zorich.  The reader can check that the $4$-slits mapping $\tilde{P}(z)=z(1-1/z^2)^{2/3}(1+1/z^2)^{1/3}$ does have the left-right symmetry and the composition  $\Se(\tilde{P}_{\epsilon}^{-1})$ does realize the Kontsevich-Zorich breaking up a double zero.  

The above splitting a double zero is an example of a general zero splitting.  In particular, for positive integers $m<n$, the special mapping
$$
Q(z) \,=\,z(1+\frac{1}{z})^{\frac{2(n-m)}{n}}(1-\frac{1}{z})^{\frac{2m}{n}}
$$   
is the Riemann mapping from the exterior of the unit disc to the complement of equal length slits, each at angle $\pi m/n$, with the positive axis.  For the scaling $Q_{\epsilon}(z)=\epsilon Q(z/\epsilon)$, the composition $S_{\epsilon}(Q_{\epsilon}^{-1}(w))$ maps the complement of equal length slits at angles $\pm \pi m/n$ with the positive real axis, to the complement of $\{\zeta\mid -2\epsilon\le \Re \zeta\le 2\epsilon,\Im\zeta=0\}$.  The composition $S_{\epsilon}(Q_{\epsilon}^{-1})$ is a family in $\epsilon$. The first and second variation cocycles can be computed as above. 

We consider the geometric description of the splitting map $\Se(Q_{\epsilon}^{-1})$.   The range of $Q_{\epsilon}$ has a positive slit $\bf{p}$ where $\Im \zeta\ge0$ and a negative slit $\bf{n}$ where $\Im \zeta\le 0$.  
Denoting the upper and lower edges of each slit by $\pm$s, we have the quadrilateral of consecutive edges $\bf{p^-}, \bf{p^+}, \bf{n^-}, \bf{n^+}$ maps by $Q_{\epsilon}^{-1}$ to the radius $\epsilon$ circle and by $\Se$ to the slit $\{\zeta\mid -2\epsilon\le\Re \zeta\le 2\epsilon,\Im\zeta=0\}$.  The mapping identifies $\bf{p^+}$ with $\bf{n^-}$ and $\bf{p^-}$ with $\bf{n^+}$.  Points symmetric with respect to the real axis are identified. The Abelian differential $\omega=-w^{n-1}dw$ is conjugated by the reflection in the real axis and is positive on the slits and the rays $\arg z=\pm \pi m/n$. By Schwarz reflection, it follows that the pushforward of $\omega$ by $\Se(Q_{\epsilon}^{-1})$ extends to be holomorphic on $\mathbb C$ except possibly at the images of the vertices of the quadrilateral.  The local form of the map is $v^m=u^n$ at the vertex of the sector with angle $2\pi m/n$.  Accordingly $mv^{m-1}dv=nu^{n-1}du$ and at the image the pushforward of $\omega$ is holomorphic with a zero of order $m-1$.  At the tips of the slits, the local form of the map is $v^2=u$ (a local disc maps to a local half disc). Accordingly $2vdv=du$ and at the image the pushforward is holomorphic with a simple zero at the image point (the slit tips are identified by the map).  Finally at the vertex of the $2\pi(n-m)/n$ sector the local form is $v^{n-m}=u^n$.  Accordingly $(n-m)v^{n-m-1}dv=nu^{n-1}du$ and at the image the pushforward of $\omega$ is holomorphic with a  zero of order $n-m-1$.   The map $\Se(Q_{\epsilon}^{-1})$ pushes forward $\omega$ to an Abelian differential with zeros of orders $n-m-1, 1$ and $m-1$ respectively at the left, middle and right points of the slit  $\{\zeta\mid-2\epsilon\le\Re \zeta \le 2\epsilon, \Im\zeta=0\}$.  The map $\Se(Q_{\epsilon}^{-1})$ {\em splits an order $n$ zero}.  The local construction in a small disc is based on using the coordinate for the normal form of the differential as the coordinate for the map.  The periods of an Abelian differential are not changed by a local construction in a disc; the deformation is isoperiodic.  Since the slits are a null set for the Hermitian square of $\omega$, the integral norm of $\omega$ is given by integrating over the slit complement and the norm is not changed by the mapping.  For a compact Riemann surface, the second observation also follows from the first by the Riemann bilinear relations. 

The third composition to consider is $\Se(\Ke^{-1}(w))$ for the scaling of maps (\ref{koebemap}) and (\ref{skeanmap}).  The composition maps the complement of a horizontal-vertical skean to the complement of $\{\zeta\mid -2\le \Re \zeta\le 2,\Im w=0\}$.  The composition is a family in $\epsilon^2$ since $\Se$ and $\Ke$ are each families in $\epsilon^2$.  Using the binomial series for $(1-\epsilon^2e^{2i\theta}/z^2)^{1/2}$ and $(1-\epsilon^2e^{-2i\theta}/z^2)^{1/2}$, we find the expansion
$$
\Ke(z)\,=\,z\,-\,\frac{\epsilon^2\cos 2\theta}{z}\,+\,\frac{\epsilon^4(1-\cos 4\theta)}{4z^3}\,+\,O(\epsilon^6).
$$
For the initial $\epsilon^2$ variations we have
$$
\dot{\Ke}(z)\,=\,-\frac{\cos 2\theta}{z}\quad\mbox{and}\quad\ddot{\Ke}(z)\,=\,\frac{(1-\cos 4\theta)}{2z^3}.\
$$
By formulas (\ref{inv1}) and (\ref{inv2}) the initial variations of $\Ke^{-1}$ are
$$
\dot{\Ke^{-1}}(w)\,=\,\frac{\cos2\theta}{w}\quad\mbox{and}
\quad\ddot{\Ke^{-1}}(w)\,=\,\frac{-1+\cos4\theta-4\cos^22\theta}{2w^3}.
$$
Recalling the initial variations now in $\epsilon^2$, $\dot{\Se}=1/z$ and $\ddot{\Se}=0$, applying Proposition \ref{elem}, we find the first and second variation cocycles are
\begin{multline}\label{SKvar}
\dot{\Se(\Ke^{-1}(}w))\,=\,\frac{1+\cos2\theta}{w} \frac{d}{dw}\quad\  \mbox{and}\\  \ddot{\Se(\Ke^{-1}(}w))\,=\,\frac{-1-4\cos2\theta+\cos4\theta-4\cos^22\theta}{2w^3} \frac{d}{dw}.
\end{multline}

We consider the geometric description of the mapping.  The maps $\Se$ and $\Ke$ are symmetric with respect to reflections in the real and imaginary axes.  It follows that the composition $\Se(\Ke^{-1})$ is also symmetric with respect to the reflections in the axes.  We do not use the additional symmetry of $\Ke$: $z\rightarrow iz,\ \theta\rightarrow \pi/2-\theta$.  The symmetries provide that the horizontal skean tips map to the horizontal slit tips and the vertical skean tips map to the origin.   Relatedly the four complementary sector vertices map to a pair of points on the slit, symmetric with respect to the origin.  The Abelian differential $\omega=wdw$ is real on the skean and is conjugated by the reflections.  By Schwarz reflection, the pushforward of $\omega$ by $\Se(\Ke^{-1})$ extends to be holomorphic on $\mathbb C$ except possibly at the images of the origin and the images of the skean tips.  At a horizontal tip the local form of the map is $v=u$ and the pushforward is holomorphic and nonzero. At a vertical tip the local form of the map is $v^2=u$.  Accordingly $2vdv=du$ and at the image of a vertical tip the pushforward of $\omega$ is holomorphic with a simple zero.  The local form of the map is $v=u^2$ at the vertex of a sector (a local quarter disc maps to a local half disc). Accordingly $dv=2udu$ and at the image of a sector vertex the pushforward of $\omega$ is holomorphic and nonzero.  The map $\Se(\Ke^{-1})$ pushes forward $\omega$ to an Abelian differential with a simple zero at the origin.  The local construction in a small disc is based on using the coordinate for the normal form of the differential as the coordinate for the map.  The periods of an Abelian differential are not changed by a local construction in a disc; the deformation is isoperiodic.  Since the slits are a null set for the Hermitian square of $\omega$, the integral norm of $\omega$ is given by integrating over the slit complement and the norm is not changed by the mapping.  For a compact Riemann surface, the second observation also follows from the first by the Riemann bilinear relations.   

The deformation matches the simplest case of the  Kontsevich-Zorich construction at a zero \cite[Section 4.2, Figure 2]{KontZor}.  In particular start with a disc neighborhood of a simple zero tiled by the four quadrants.  The quadrants are four Kontsevich-Zorich half discs.  Referring to the $\Se(\Ke^{-1})$ mapping intermediate segments on the unit circle, the map $\Se$ identifies the circular segments $[\theta, \pi]$ and $[-\pi/2,-\theta]$ (in reverse order) and $[\pi/2,\pi-\theta]$ with $[\pi+\theta,3\pi/2]$ (in reverse order).  The $\Se$ images are horizontal segments.  The four segments have $\Ke$ images on the vertical boundaries of the quadrants. In summary vertical boundary segments are reglued to horizontal boundary segments in the manner of Kontevich-Zorich. 

\section{Variation of Green's functions and Riemann period matrices}\label{sec4}

We follow Schiffer's analysis for the variation of Green's functions and Abelian differentials for  compact surfaces \cite[Section 7.8]{SchS1}.  We recall his treatment and use the setup to develop second order deformation expansions. Begin with $d\Omega_{q_0q_1}$ the Abelian differential of the third kind, periods with vanishing real parts, with a pole of residue $-1$ at $q_0$ and $+1$ at $q_1$ \cite[Section 4.1]{SchS1}.  Let $\Omega_{q_0q_1}$ be the indefinite integral, a multivalued holomorphic function with leading term $-\log(z(p)-z(q_0))$ near $q_0$ and $\log(z(p)-z(q_1))$ near $q_1$ for a generic local coordinate.  The multivalues differ by imaginary values.  Define the {\em double pole Green's function} \cite[page 98]{SchS1},
\begin{equation}
V(p,p_0;q,q_0)\,=\,\Re\{\Omega_{qq_0}(p)\,-\,\Omega_{qq_0}(p_0)\}.
\end{equation}
The Green's function is real harmonic, symmetric in the pairs $(p,p_0)$ and $(q,q_0)$, and anti symmetric in each of $(p,p_0)$ and $(q,q_0)$.  The exponential $e^{V(p,p_0;q,q_0)}$ has the Arakelov theory interpretation as a metric for the degree zero line bundle $\mathcal O(q_1-q_0)$ normalized to unity at $p_0$.   
 Also consider the holomorphic Abelian kernel \cite[Section 4.3, (4.2.25)]{SchS1},
\begin{equation}\label{LamV}
\Lambda(p,q)\,=\,-\frac{2}{\pi}\frac{\partial^2V(p,p_0;q,q_0)}{\partial p\partial q} dpdq\,=\,-\frac{1}{\pi}\frac{\partial^2\Omega_{qq_0}(p)}{\partial p\partial q}dpdq.
\end{equation}
In Schiffer's notation $p,q$ may represent points or may represent generic variables. The quantity $\Lambda$ is a symmetric complex tensor.   If the points $p,q$ lie in a common coordinate $z$, then $\Lambda$ has a coordinate expansion
\begin{equation*}
\Lambda(p,q)\,=\,\big(\frac{1}{\pi(z(p)-z(q))^2}\ +\ regular \ holomorphic\big)\,dz(p)dz(q).
\end{equation*}
The expansion can be used to define the regularization at $p$,
$$
\Lambda^{reg,z}(p)\,=\,\lim_{q\rightarrow p}\big(\Lambda(p,q)\,-\,\frac{dz(p)dz(q)}{\pi(z(p)-z(q))^2}\,\big).
$$
The regularization is a locally defined holomorphic quadratic differential.  

We describe bases for the space of Abelian differentials.  Let $K_1,\dots,K_{2g}$ be a canonical homology basis for the genus $g$ surface $R$, with odd elements the {\em A cycles} and even elements the {\em B cycles}.  In particular for $1\le \mu\ne \nu\le g$, the intersection relations are $K_{2\mu-1}\cdot K_{2\mu}=1$ and  $K_{2\mu-1}\cdot K_{2\nu}= K_{2\mu}\cdot K_{2\nu}=0$.  A standard normalized basis $\{\alpha_{\mu}\}$ for the Abelian differentials is defined by the condition $\int_{K_{2\mu-1}}\alpha_{\nu}\,=\,\delta_{\mu\nu},$ $1\le\mu,\nu\le g$, for the Kronecker delta, and the corresponding Riemann period matrix $\Pi_{\mu\nu}$ has entries $\int_{K_{2\mu}}\alpha_{\nu}$, for $1\le\mu,\nu\le g$, \cite{GHbook}. Results in the literature are often given in terms of such bases with normalized {\em A periods}.  The given Riemann period matrix is symmetric with positive definite imaginary part.  The Riemann bilinear relations provide that 
$$\frac{i}{2}\int_{\mathcal M} \alpha_{\mu}\wedge\overline{\alpha_{\nu}}= \,\Im\Pi_{\mu\nu}.
$$  
Colombo and Frediani present their results in terms of an orthonormal basis for the pairing $i\int_{\mathcal M}\alpha\wedge\overline{\beta}$, \cite{ColFr}. In some contrast, Schiffer uses the differentials dual in the pairing to the period functionals.  In particular, he considers the $2g$ differentials $dZ_{K_{\mu}}$ defined by
$$
\int_{K_{\mu}}\omega\,=\,-\frac{i}{2}\int_{\mathcal M}\omega\wedge\overline{dZ_{K_{\mu}}},
$$
for all Abelian differentials $\omega$ \cite[(3.1.20) and Section 3.3]{SchS1},  and defines a $2g\times 2g$ Hermitian Riemann period matrix by 
$$
\Gamma_{\mu\nu}\,=\,\frac{i}{2}\int_{\mathcal M}dZ_{K_{\mu}}\wedge\overline{dZ_{K_{\nu}}},
$$
\cite[(1.5.18) and (3.2.8)]{SchS1}. The setup is well suited to Schiffer's approach, since the differentials are given by integrals of the Abelian kernel \cite[(4.3.1)]{SchS1},
\begin{equation}\label{lamome}
dZ_{K_{\mu}}(q)\,=\,\int_{K_{\mu}}\Lambda(p,q).
\end{equation}
In particular, the Abelian differentials and period matrix are determined from the central quantity the double pole Green's function.  

We follow Schiffer's approach.  The defining property for each basis $\{\alpha_{\mu}\}, \{dZ_{K_{2\mu-1}}\}$ and $\{dZ_{K_{2\mu}}\}$ leads to  change of bases formulas. An orthonormal basis for the pairing $\frac{i}{2}\int\alpha\wedge\overline{\beta}$ is given as 
$$
\big(\Im \Pi\big)^{-1/2}_{\mu\nu}\big(\alpha_{\nu}\big)
$$
and change of bases as 
$$
\big(dZ_{2\mu-1}\big)\,=\,-\big(\Im\Pi\big)^{-1}_{\mu\nu}\big(\alpha_{\nu}\big)\quad\mbox{and}\quad  \big(dZ_{2\mu}\big)\,=\,-\big(\overline{\Pi}\big)_{\mu\nu}\big(\Im\Pi\big)^{-1}_{\nu\sigma}\big(\alpha_{\sigma}\big),
$$
for $1\le\mu,\nu,\sigma\le g$ and $(dZ_*)$,  $(\alpha_*)$ the appropriate column vectors of differentials.  The matrix $\Gamma_{2\nu-1\,2\sigma-1}$ represents the A periods of the basis $(dZ_{2\nu-1})$ and so is invertible.  We present three relations involving the basis.  The first relation $(\alpha_{\nu})=-(\Gamma)^{-1}_{2\nu-1\,2\sigma-1}(dZ_{2\sigma-1})$ is verified by evaluating the  A periods, the $K_{2\tau-1}$ integrals.  The second relation  $\Pi_{\nu\tau}=(\Gamma)^{-1}_{2\nu-1\,2\sigma-1}\Gamma_{2\sigma-1\,2\tau}$ is verified by evaluating the B periods, the $K_{2\tau}$ integrals, for the first relation.  The third relation $
(dZ_{2\mu})\,=\,\Gamma_{2\mu\,2\nu-1}(\Gamma)^{-1}_{2\nu-1\,2\sigma-1}(dZ_{2\sigma-1})$ 
is verified by evaluating the A periods, the $dZ_{2\tau-1}$ integrals. The last relation can also be written as $(dZ_{2\mu})^T=(dZ_{2\sigma-1})^T\overline{\Pi}_{\sigma\mu}$, since the matrix $\Gamma_*$ is Hermitian.  
In the following, we simplify the notation by writing $\omega_{\mu}=dZ_{K_{\mu}}$.

We proceed with the setup for the variation of the Green's function.  The construction is in terms of a given local coordinate $z$ with domain $U$ and a curve $\gamma$ bounding a disc in the domain of $z$.  
Let $r(z)$ be a holomorphic function in $z$ with domain a neighborhood of $\gamma$.     Provided $r(z)$ is suitably small, the deformation cocycle defines a new Riemann surface $R^*$ given by attaching the exterior of $\gamma$ (the complement of the disc bound by $\gamma$) to the interior of $\gamma^*=\gamma+r(\gamma)$ by identifying $z(p)$ on $\gamma$ to $z(p)+r(z(p))$ on $\gamma^*$.  Let $U_0$ be a disc subset of the domain of $z$ containing both $\gamma$ and $\gamma^*$.  The complement $R_0$ of $U_0$ is a common subdomain of $R$ and $R^*$.   In the following, quantities for $R^*$ are denoted by a $^*$.  

Schiffer's analysis is based on an exact formula for the variation of the Green's function.   Expansions are obtained from the exact formula.

\begin{theorem}\label{exactform}
	\textup{\cite[Theorem 7.5.1]{SchS1}.} Notation as above.  For points $p,p_0,q,q_0$ in the subsurface $R_0$, the Green's functions and Abelian differentials satisfy
	\begin{equation}
	\label{mainid}
	V^*(p,p_0;q,q_0)\,-\,	V(p,p_0;q,q_0)\,=\,\Re\bigg\{\frac{1}{2\pi i}\int_{\partial R _0}\Omega_{qq_0}(t)d\Omega_{pp_0}^*(t) \bigg\}.
	\end{equation}
\end{theorem}
The proof combines the fundamental solution property of the Green's function, the residue theorem and the vanishing real periods for the differentials $d\Omega$.  

We are ready for the expansion.   Consider that the deformation cocycle defining $R^*$ has an expansion $r(z,\epsilon)=\epsilon\, r_1(z)+\frac{\epsilon^2}{2}\,r_2(z)+O(\epsilon^3)$ in a parameter $\epsilon$ with each term a vector field. The following generalizes Schiffer's first order expansion \cite[(7.8.8)]{SchS1}. The quantities involved are harmonic or holomorphic.  Supermum bounds immediately give rise to $C^k$ bounds and interchanging differentiation and integration is a straightforward matter.    

\begin{theorem}
	\label{main} Notation as above.  For points $p,p_0,q,q_0$ in the subsurface $R_0$, the Green's functions and Abelian differentials satisfy
			\begin{multline*}			
		V^*(p,p_0;q,q_0)\,-\,V(p,p_0;q,q_0)\,=\,\Re\bigg\{\frac{1}{2\pi i}
		\int_{\gamma} \epsilon\, r_1(t)d\Omega_{qq_0}(t)d\Omega_{pp_0}(t) \\
		+\,\frac{\epsilon^2}{2}\bigg(r_2(t)d\Omega_{qq_0}(t)d\Omega_{pp_0}(t)\,+\,r_1(t)^2d\Omega_{pp_0}(t)\frac{\partial}{\partial t}d\Omega_{qq_0}(t) \\
		+\,2i\int_{\tilde{\gamma}}r_1(s)\Lambda(s,t)d\Omega_{pp_0}(s)\,r_1(t)d\Omega_{qq_0}(t)\bigg)
				 \bigg\}\,+\,O(\epsilon^3),
			\end{multline*}
		where $\tilde{\gamma}$ is a curve homologous to $\gamma$ and in its interior\footnote{The present formula differs in sign from \cite[(7.8.8)]{SchS1}.  In transitioning from Theorem 7.5.1 to formula (7.8.4), Schiffer replaces $\partial\mathcal M_0$ with $\gamma$ without a sign change.  The region $\mathcal M_0$ is exterior to $\gamma$ and so the curve in (7.8.4) is negatively oriented.  Accordingly starting with (7.8.4) our formulas differ from Schiffer's by a sign. }.  The $\gamma$ integral is in the variable $t$ and the $\tilde{\gamma}$ integral is in the variable $s$.
\end{theorem}  
\begin{proof}
	The approach is to develop an expansion for the right hand side of (\ref{mainid}).   The curve $\gamma\epp=\gamma+r(\gamma,\epsilon)$ on $R^*$ bounds a topological disc. The structure sheaf $\mathcal O$ of $R^*$ is characterized in terms of germs on the exterior of $\gamma$ and the interior of $\gamma\epp$ that are holomorphic upon the identification $z\rightarrow z+r(z,\epsilon)$.  The quantities $\Omega_{qq_0}$ and $d\Omega^*_{pp_0}$ are holomorphic on the interior of $\gamma_{\epsilon}$. By Cauchy's Theorem we have that
	$$
	\frac{1}{2\pi i}\int_{\gamma\epp}\Omega_{qq_0}(t_1)d\Omega^*_{pp_0}(t_1)\,=\,0.
	$$
	  By definition for $t_1=t+r(t,\epsilon)$ we have $d\Omega_{pp_0}^*(t_1)=d\Omega_{pp_0}^*(t)$ for $t$ on $\gamma$.  We can now rewrite the integral as
	$$
	\frac{1}{2\pi i}\int_{\gamma}\Omega_{qq_0}(t+r(t,\epsilon))d\Omega^*_{pp_0}(t)\,=\,0.
	$$
	Taylor's Theorem  provides an expansion for the first factor of the integrand
	\begin{multline*}
	\Omega_{qq_0}(t+r(t,\epsilon))\,=\,\Omega_{qq_0}(t)\,+\,\epsilon r_1(t)\Omega_{qq_0}'(t)\\+\,\frac{\epsilon^2}{2}\big(r_2(t)\Omega_{qq_0}'(t)\,+\,r_1(t)^2\Omega_{qq_0}''(t)\big)\,+\,O(\epsilon^3).
	\end{multline*}
	We now add the integral to (\ref{mainid}) (see the footnote regarding the orientation of the integration curve) and substitute the Taylor expansion to find\footnote{The expansion is valid to the same order as the Taylor expansion.}  
	\begin{multline}\label{Vexp}
	V^*(p,p_0;q,q_0)\,-\,V(p,p_0;q,q_0)\,=\,\Re\bigg\{ \frac{1}{2\pi i} \int_{\gamma}\big( \epsilon r_1(t)d\Omega_{qq_0}(t)\\ +\,\frac{\epsilon^2}{2}\big(r_2(t)d\Omega_{qq_0}(t)\,+\,r_1(t)^2\frac{\partial}{\partial t}d\Omega_{qq_0}(t)\big)\big)d\Omega^*_{pp_0}(t)\bigg\}    \,+\,O(\epsilon^3).
	\end{multline}
	The final step is to use this preliminary expansion to expand for the differential $d\Omega^*_{pp_0}(t)$ on the right hand side\footnote{Schiffer's exact formula (7.8.6) provides for a complete expansion in $\epsilon$ by using the formula and relation  $d\Omega^*=2\frac{\partial}{\partial t}V^*$ to substitute iteratively for $d\Omega^*$ on the right hand side.}.  We use the symmetry of $V$ and write
	 ${\Omega_{pp_0}^*}'(t)=2\frac{\partial}{\partial t}V^*(t,t_0;p,p_0)=2\frac{\partial}{\partial t}V^*(p,p_0;t,t_0)$ and apply $2\frac{\partial}{\partial q}$ to the preliminary expansion and note (\ref{LamV}) to find
	 $$
	 d\Omega_{pp_0}^*(t)\,=\,d\Omega_{pp_0}(t)\,+\,\frac{i\epsilon}{2} \int_{\tilde{\gamma}} r_1(s)\Lambda(s,t)d\Omega_{pp_0}(s)\,+\,O(\epsilon^2),
	 $$
	 where the curve $\tilde{\gamma}$ is taken inside of $\gamma$ to ensure that the expansion is valid on $\gamma$.  The argument is complete.  	 
\end{proof}

Following Schiffer we combine (\ref{lamome}) and the expansion for the Green's function to obtain a variational formula for the period matrix.  The formulas relating different period matrices or bases can be used to find the variational formula for $\Pi_{\mu\nu}$.  

\begin{corollary}
	\label{derivperi}  Notation as above.  The Riemann period matrix satisfies
	\begin{multline*}
		\Gamma_{\mu\nu}^*\,-\,\Gamma_{\mu\nu}\,=\,\frac{\epsilon}{4}\int_{\gamma}r_1(t)\omega_{\mu}(t)\omega_{\nu}(t)\,+\, 
		\frac{\epsilon^2i}{8}\int_{\gamma}r_1(t)\omega_{\nu}(t)
		\int_{\tilde{\gamma}}r_1(s)\Lambda(s,t)\omega_{\mu}(s)\\
		+\,\frac{\epsilon^2}{8}\int_{\gamma}
		r_2(t)\omega_{\mu}(t)\omega_{\nu}(t)\,+\,r_1(t)^2\omega_{\mu}(t)\frac{\partial}{\partial t}\omega_{\nu}(t)\,+\,O(\epsilon^3).
			\end{multline*}
	\end{corollary}
	\begin{proof}  The expansion is a transformation of the Green's function expansion as follows. First apply $-\frac{2}{\pi}\frac{\partial^2}{\partial p\partial q}$ to the expansion and use the definition (\ref{LamV}) to obtain an expansion for the Abelian kernel $\Lambda$. 
	Next integrate $p$ over $K_{\mu}$, applying the property (\ref{lamome}), to obtain an expansion for the basis differentials $\omega_{\mu}$.  Finally integrate $q$ over $-K_{\nu}$, applying the definition $-\frac{i}{2}\int_{K_{\nu}}\omega_{\mu}=\Gamma_{\mu\nu}$ on the left to obtain periods and the property (\ref{LamV}) on the right to obtain the differentials $\omega_{\nu}$.   
	\end{proof}  
	
If the cocycles are rational functions with poles interior to $\gamma$, then the integrals are evaluated by products of values of the Abelian differentials and their derivatives at the poles.  Schiffer gives the first variation in his formula (7.8.15).  The first variation of $\Pi_{\mu\nu}$ is Rauch's formula \cite{Rauch}.

\begin{corollary}\label{formu}
	Notation as above.  Let $p$ be a point with a local coordinate $z$ with $z(p)=0$.  For constants $a,b\in\mathbb C$, $m\in\mathbb N$, let $r(z,\epsilon)=(\epsilon\frac{a}{z} +\frac{\epsilon^2}{2}\frac{b}{z^m} )\frac{d}{dz}\,+\,O(\epsilon^3)$ be a deformation cocycle for a punctured neighborhood of $p$.   The Riemann period matrix satisfies 
	\begin{multline*}\label{mainform}
	\Gamma_{\mu\nu}^*\,-\,\Gamma_{\mu\nu}\,=\,\frac{\epsilon\pi i}{2} a\,\omega_{\mu}(p)\omega_{\nu}(p)\,-\,\frac{\epsilon^2\pi^2i}{2}a^2\lambda(p)\omega_{\mu}(p)\omega_{\nu}(p)\\
	+\frac{\epsilon^2\pi i}{4}\bigg(a^2\omega_{\mu}'(p)\omega_{\nu}'(p)\,+\,b\frac{1}{(m-1)!}\frac{d^{m-1}}{dz^{m-1}}\big(\omega_{\mu}(z)\omega_{\nu}(z)\big)\big\vert_{z=p}\bigg)\,+\,O(\epsilon^3),
	\end{multline*}
	where the basis differentials $\omega_{\mu}$ and their derivatives are evaluated in the coordinate $z$ and $\lambda(z)dz^2$ is the Abelian kernel regularization $\Lambda^{reg,z}$.  
\end{corollary}

\begin{proof}  We consider the above period matrix expansion with the curves $\gamma$ and $\tilde{\gamma}$ in the coordinate chart.  The integrands are holomorphic except at the point $p$.  The integrals are evaluated by computing residues at $p$. To that purpose, write 
	$$
	\Lambda(s,t)\,=\,\frac{dsdt}{\pi(t-s)^2}\,+\,\tilde{\Lambda}(s,t),
	$$ 
	where the quantities are in terms of the given coordinate and $\tilde{\Lambda}$ is a regular holomorphic $dsdt$ tensor.  By the  residue calculation  
	$$
	\int_{\tilde{\gamma}}\frac{\omega_{\mu}(s)ds}{s(t-s)^2}\,=\,2\pi i\,\frac{\omega_{\mu}(p)}{t^2},
	$$
	since $1/(t-s)^2$ is holomorphic in $s$ on the interior of $\tilde{\gamma}$.   The remaining integrals are calculated by substituting Taylor expansions in terms of the coordinate and calculating residues. 
\end{proof}

Let $R$ be a compact Riemann surface of genus $g\ge1$. The cotangent space of the deformation space at $R$ is the $1$ dimensional for $g=1$ and in general $3g-3$ dimensional space of holomorphic quadratic differentials.   We consider the moduli space of equivalence classes of pairs $(R,\omega)$ where $R$ is a homotopy marked surface and $\omega$ is an Abelian differential.  The family of pairs is the holomorphic Hodge bundle over the Teichm\"{u}ller space.  The fiber is $A(R)$ the space of Abelian differentials on $R$.  We are interested in the isoperiodic foliation with leaves the Abelian differentials with given periods.  As considered in Section \ref{4var}, 
given an Abelian differential the deformations corresponding to the maps $e_n^{-1}\Ne(e_n\Ne^{-1})$, $\Se(\Pe^{-1})$ and $\Se(\Ke^{-1})$ in a neighborhood of a zero of $\omega$ are isoperiodic deformations.  Each construction requires the coordinate for the normal form of the differential. The deformations are isoperiodic since the deformation is local in a small disc. In consequence the first and second cocycles of (\ref{NNvar}), (\ref{SPvar}) and (\ref{SKvar}) evaluate to zero in the formula of Corollary \ref{formu}.   

We now find that the Schiffer variations at the zeros of an Abelian differential span the tangent space of the isoperiodic level set of the differential in the deformation space. Start with a non trivial  differential $\omega$ and consider that $g>1$.   For an order $m$ zero at a point $p$, consider a general local coordinate $z$ with $z(p)=0$ and the first variation cocycles
$$
\frac{1}{z}\frac{d}{dz},\,\dots\,,\frac{1}{z^m}\frac{d}{dz}.
$$
We refer to these cocycles as the {\it Schiffer deformations of the zero} at $p$. A differential has a total of $2g-2$ Schiffer deformations at zeros.  Also for a  general point $q$ with general local coordinate $z(q)=0$, consider the first variation cocycles
$$
\frac{1}{z}\frac{d}{dz},\,\dots\,,\frac{1}{z^g}\frac{d}{dz}.
$$
We consider the Wronskian of a basis of Abelian differentials computed in a local coordinate.  Since the differentials are linearly independent, the determinant of the matrix
$$
\bigg(\frac{d^k}{dz^k}\omega_{2\sigma-1}\bigg)_{1\le\sigma\le g,\,0\le k\le g-1}
$$
is a non trivial local holomorphic function.  It is a classical observation that the determinant is a section of the canonical bundle of the surface to the power $\frac{(g+1)g}{2}$ \cite{GunRS}.  The  matrix is generically non singular.  
\begin{proposition}\label{defbasis}
	For $g>1$ and a non trivial differential $\omega$, a basis for the infinitesimal deformations of the surface is given by any $2g-3$ Schiffer deformations of (possibly multiple) zeros of $\omega$ and $g$ Schiffer deformations at a general point.  Any $2g-3$ Schiffer deformations at zeros of $\omega$ span the tangent space of the isoperiodic level.  
\end{proposition} 
 
  \begin{proof} 
  	The first statement follows on showing that the annihilator in the space of holomorphic quadratic differentials of the collection of cocycles is the trivial subspace.  Consider a holomorphic quadratic differential $\varphi$ annihilated by $2g-3$ deformations of zeros of $\omega$.  By a residue calculation the differential $\varphi$ vanishes at the $2g-3$ zeros.   The quotient $\varphi/\omega$ is an Abelian differential with a possible simple pole at the remaining zero of $\omega$.  An Abelian differential cannot have a single simple pole by Riemann Roch.  The quotient is holomorphic and $\varphi=\omega\alpha$ for $\alpha$ a holomorphic Abelian differential.  We next consider the $g$ Schiffer deformations at a general point.  Choose a point such that $\omega$ is non zero and the Wronskian of an Abelian differentials basis is non singluar.  Let $z$ be a local coordinate such that $\omega=dz$ in a neighborhood.  Since $\omega$ has unit coefficient in the local coordinate, we have the equality of Wronskians
  	\begin{equation}
  	\label{Wronsk}
  	  	\bigg(\frac{d^k}{dz^k}\omega\omega_{2\sigma-1}\bigg)_{1\le\sigma\le g,\,0\le k\le g-1}\,=\,\bigg(\frac{d^k}{dz^k}\omega_{2\sigma-1}\bigg)_{1\le\sigma\le g,\,0\le k\le g-1}
  	\end{equation}
  	computed in the local coordinate $z$.  It follows that only the  product of  $\omega$ and the trivial  Abelian differential is annihilated by the $g$ Schiffer deformations.  The collection of $3g-3$ cocycles is a basis. 
  	
  	We consider the isoperiodic condition. Relative to a given homology basis, the A periods $(a_{\mu})$ of $\omega$ prescribe the linear combination of the normalized basis $\omega=(a_{\mu})^T(\alpha_{\mu})=-(a_{\mu})^T(\Im\Pi)_{\mu\nu}(\omega_{2\nu-1})$. For $(\Im\Pi)_*$ and $(\omega_*)$ depending on a variation parameter $\epsilon$, the {\em constant} A periods condition provides the variational equation $(a_{\mu})^T\dot{(\Im\Pi)}_{\mu\nu}\Gamma_{2\nu-1\,2\sigma-1}+(a_{\mu})^T(\Im\Pi)_{\mu\nu}\dot{\Gamma}_{2\nu-1\,2\sigma-1}=0$.  Similarly the variation of the B periods of $(a_{\mu})^T(\alpha_{\mu})$ is given in terms of   $(a_{\mu})^T\dot{(\Im\Pi)}_{\mu\nu}\Gamma_{2\nu-1
  	\,2\sigma}+(a_{\mu})^T(\Im\Pi)_{\mu\nu}\dot{\Gamma}_{2\nu-1\,2\sigma}$.  The period matrix $\Gamma_{2\sigma-1\,2\nu-1}$ is nonsingular and we can solve the first equation for $(a_{\mu})^T\dot{(\Im\Pi)}_{\mu\nu}$  to substitute into the second expression.  We apply the relations for period matrices to find a formula for the variation of the B periods 
  	$$
  	-\big(a_{\mu}\big)^T\big(\Im\Pi\big)_{\mu\tau}\dot{\Gamma}_{2\tau-1\,2\sigma-1}\Pi_{\sigma\rho}
    \,+\,\big(a_{\mu}\big)^T\big(\Im\Pi\big)_{\mu\nu}\dot{\Gamma}_{\,2\nu-1\,2\rho}.
  	$$
  	Since the first variation term of Corollary \ref{formu} is a bilinear functional of the  differentials $\omega_{\mu}$ and $\omega_{\nu}$, we can use 	
   the change of basis $(\omega_{2\rho})^T=(\omega_{2\sigma-1})^T\overline{\Pi}_{\sigma\rho}$  to write $\dot{\Gamma}_{2\nu-1\,2\rho}=\dot{\Gamma}_{2\nu-1\,2\sigma-1}\overline{\Pi}_{\sigma\rho}$. 
   Combining contributions we obtain the formula for the variation of B periods
  	$$
  	-2i\big(a_{\mu}\big)^T\big(\Im\Pi\big)_{\mu\tau}\dot{\Gamma}_{2\tau-1\,2\sigma-1}
  	\big(\Im \Pi\big)_{\sigma\rho}.
  	$$
  	
  	Since the formula of Corollary \ref{formu} is linear in each  differential, the quantities
  	$(a_{\mu})^T(\Im\Pi)_{\mu\tau}\dot{\Gamma}_{2\tau-1\,2\sigma-1}$ can be formally treated as the first order variation of the period of $\omega$ on the cycle $K_{2\sigma-1}$; the quantities are evaluated by applying  Schiffer variations to $\omega\omega_{2\sigma-1}$.  The quantities are then multiplied on the right by a nonsingular matrix.   
  	
  	We now apply the formula.  The Schiffer deformations at zeros of $\omega$ immediately give trivial evaluations for the quantities  $(a_{\mu})^T(\Im\Pi) _{\mu\tau}\dot{\Gamma}_{2\tau-1\,2\sigma-1}$.  In particular, the $2g-3$ deformations of zeros of $\omega$ are tangent to the isoperiodic locus.  For the general $g$ Schiffer deformations, we again use the setup with local coordinate $z$ with $\omega=dz$.  By Corollary \ref{formu}, the evaluation of the $g$ Schiffer deformations for the B periods are given  by (\ref{Wronsk}) right multiplied by a nonsingular matrix. Since the Wronskian is non singular, only the trivial combination of $g$ Schiffer deformations has trivial variation of the periods.   
  	\end{proof}
  	
  	\section{Variation of Abelian differentials}\label{sec5}
  	We use Schiffer's approach to give a complete expansion for the variation of an Abelian differential.  If the deformation cocycle has coefficient a rational function then the terms of the expansion are explicitly evaluated in terms of values and derivatives of the initial Abelian differential and the Abelian kernel at the appropriate poles.  The present expansion can be compared to the expansions of Yin \cite{Yin}, Zhao-Rao \cite{ZhRa} and Liu-Zhao-Rao \cite{LiZhRa}.  These authors use $\bar{\partial}$-methods and give expansions with iterated integrals of the Green's function acting on one-forms.  Their expansions do not immediately give differentials dual to cycles.  Schiffer's focus on the Green's function $V$ and Abelian kernel $\Lambda$ immediately gives differentials dual to cycles. Calculation of the Riemann period matrix is immediate from the differentials dual to the {\em A cycles}.  
  	
  	\begin{proposition} \label{omegform} Notation as above.  For a deformation cycle $\epsilon r(z)\frac{d}{dz}$, the variation of the Abelian differential $\omega_{\mu}$ is given by an asymptotic expansion 
  	\begin{multline*}	
  		\omega_{\mu}^*\,=\,\sum_{k=0}^{\infty}\epsilon^k\omega_{\mu}^{(k)},
  		\mbox{ where }\ \omega^{(0)}_{\mu}\,=\,\omega_{\mu}, \\ \mbox{and for } k \mbox{ positive, } \omega^{(k)}_{\mu}(t)\,=\,\frac{i}{2}\int_{\gamma}\,\sum_{m=1}^k\frac{r(s)^m}{m!}\big(\frac{\partial^{m-1}}{\partial s^{m-1}} \Lambda(t,s)\big)\omega_{\mu}^{(k-m)}(s).
   	\end{multline*}
   	 Each term is defined on the exterior of its integration curve. The curves are chosen to be nested; the curve for the $k^{th}$ integral is contained in the interior of the curve for the $k+1^{st}$ integral.  For $n$ non negative, the difference $\omega_{\mu}^*\,-\,\sum_{k=0}^n\epsilon^k\omega_{\mu}^{(k)}$ is bounded as $O(\epsilon^{n+1})$ for all small $\epsilon$.  
  	\end{proposition}
  	\begin{proof}  We continue the discussion for formula (\ref{Vexp}) and consider substitution of the Taylor expansion for $d\Omega_{qq_0}$ to find the formula
  	\begin{multline*}
  	V^*(p,p_0;q,q_0)\,-\,V(p,p_0;q,q_0)\,=\\ \Re\bigg\{\frac{1}{2\pi i}\int_{\gamma}\,\sum_{m=1}^{\infty}\frac{(\epsilon r(t))^m}{m!}\big(\frac{\partial^{m-1}}{\partial t^{m-1}}d\Omega_{qq_0}(t)\big)d\Omega_{pp_0}^*(t)\bigg\}.
  	\end{multline*}
  	As in Corollary \ref{derivperi}, we apply $-\frac{2}{\pi}\frac{\partial^2}{\partial p\partial q}$ to the expansion and use the definition (\ref{LamV}) to obtain a resulting expansion for the Abelian kernel $\Lambda$.  Next we integrate $p$ over $B_{\mu}$ and apply property (\ref{lamome}), to find an expansion for $\omega_{\mu}$ 
  	\begin{equation}
  	\label{omegastar}
    	\omega_{\mu}^*(q)\,-\,\omega_{\mu}(q)\,=\,\frac{i}{2}\int_{\gamma}\,\sum_{m=1}^{\infty}\frac{(\epsilon r(t))^m}{m!}\big(\frac{\partial^{m-1}}{\partial t^{m-1}}\Lambda(q,t)\big)\omega_{\mu}^*(t).
  	\end{equation}
  	The expansion provides for the $\epsilon$ expansion of $\omega^*_{\mu}$, since the right hand side immediately has positive order in $\epsilon$.  Substituting the expansion $\omega_{\mu}^*=\sum_{k=0}^{\infty}\epsilon^k \omega_{\mu}^{(k)}$ gives the desired inductive definitions for the terms $\omega_{\mu}^{(k)}$. 
  	
  	We consider bounding the approximation of $\omega^*_{\mu}$.   Choose a finite cover of coordinate charts for the surface compatible with the deformation cocycle description.   Select the charts so that relatively compact subsets, one for each chart, also provide a cover.   Norms of quantities on the surface are given by considering the norms of local coordinate representations on the relatively compact subsets. The charts and relatively compact subsets play the corresponding role for small  deformations of the surface.  
  	
  	First consider the expansion $\sum_{m=0}^{\infty}\frac{(\epsilon r(t))^m}{m!}\frac{\partial^{m-1}}{\partial t^{m-1}}\Lambda(q,t)$, which is the Taylor series for $-\frac{1}{\pi}d\Omega_{qq_0}(t+\epsilon r(t))$.   In terms of the specified norms the sum $\sum_{m=a+1}^{\infty}\frac{(\epsilon r(t))^m}{m!}\frac{\partial^{m-1}}{\partial t^{m-1}}\Lambda(q,t)$ is bounded as $O(\epsilon^{a+1})$.   We proceed by  induction to prove that $\omega_{\mu}^*-\sum_{k=0}^n\epsilon^k\omega_{\mu}^{(k)}$ is bounded as $O(\epsilon^{n+1})$. For $n=0$, the desired bound follows from (\ref{omegastar}),  the bound for the series with $n=0$ and that $\omega_{\mu}^*$ is bounded for all small $\epsilon$. We then assume the bound for a given value of $n$ and substitute the approximation for $\omega_{\mu}^*$ into the right hand side of (\ref{omegastar}).  Given the bound for the series with $a=0$, we have 
  	\begin{multline*}
  	\omega_{\mu}^*(q)\,-\,\omega_{\mu}(q)\,=\,\frac{i}{2}\int_{\gamma}\,\sum_{m=1}^{\infty}\frac{(\epsilon r(t))^m}{m!}\big(\frac{\partial^{m-1}}{\partial t^{m-1}}\Lambda(q,t)\big) \sum_{k=0}^n\epsilon^k\omega_{\mu}^{(k)}(t)    \ +\ O(\epsilon^{n+2})\\
  	=\,\frac{i}{2}\int_{\gamma}\,\sum_{k=0}^n\epsilon^k\omega_{\mu}^{(k)}(t)
  	\sum_{m=1}^{n+1-k}\frac{(\epsilon r(t))^m}{m!} \big(\frac{\partial^{m-1}}{\partial t^{m-1}}\Lambda(q,t)\big) \ +\ O(\epsilon^{n+2}),
  	\end{multline*}
  	where we have applied the bound for the series with $a=n+1,\dots,1$.   Rearranging the double sums, we find the integral 
  	$$
  	\qquad=\, \frac{i}{2}\int_{\gamma}\,\sum_{j=1}^{n+1}\epsilon^j \sum_{m=1}^j \frac{(\epsilon r(t))^m}{m!}\big(\frac{\partial^{m-1}}{\partial t^{m-1}}\Lambda(q,t)\big) \omega_{\mu}^{(j-m)}(t),
  	$$
  	which is simply $\sum_{j=1}^{n+1}\epsilon^j\omega_{\mu}^{(j)}$.  The proof is complete. 
  	\end{proof}
  	
  	We explicitly calculate the first and second variation.  From Proposition \ref{omegform}, we have for $t$ a local coordinate with $\gamma$ in the coordinate chart and $q$ outside $\gamma$,
  	$$
  	\omega_{\mu}^{(1)}(q)\,=\,\frac{i}{2}\int_{\gamma} r(t)\Lambda(q,t)\omega(t)
  	$$
  	and
  	$$
  	\omega_{\mu}^{(2)}(q)\,=\,\frac{i}{2}\int_{\gamma} r(t)\Lambda(q,t)\omega_{\mu}^{(1)}(t)\,+\,\frac{r(t)^2}{2}\big(\frac{\partial}{\partial t}\Lambda(q,t)\big)\omega_{\mu}(t).
  $$
  Let $p$ be a point inside $\gamma$ with $t(p)=0$ and let $\epsilon \frac{a}{t}\frac{d}{dt}$ be the deformation cocycle.  Again we have the expansion of Corollary \ref{formu},
  $$
  \Lambda(p,t)\,=\,\frac{dpdt}{\pi t^2}\,+\,\tilde{\Lambda}(p,t)
  $$
  for the Abelian kernel near $p$, where $\tilde{\Lambda}$ is a regular $dpdt$ tensor.  The integrals are evaluated by computing residues at $p$.  We find the formulas for the variations of $\omega_{\mu}$
  \begin{multline}
  \label{abelvar}
  \quad\quad\quad\quad \omega_{\mu}^{(1)}(q)\,=\,-\pi a\Lambda(q,p)\omega_{\mu}(p)\quad
  \mbox{and} \\
  \omega_{\mu}^{(2)}(q)\,=\,\pi^2 a^2\Lambda(q,p)\lambda(p)\omega_{\mu}(p)\,-\,\frac{\pi}{2}a^2\frac{\partial}{\partial p}\Lambda(q,p)\frac{\partial}{\partial p}\omega_{\mu}(p),
  \end{multline}
  for $q\ne p$.

 \bibliographystyle{alpha}
 \bibliography{scottbib}
 
\providecommand\WlName[1]{#1}\providecommand\WpName[1]{#1}\providecommand\Wl{Wlf}\providecommand\Wp{Wlp}\def\cprime{$'$}

\end{document}